\documentclass{article}

\usepackage{cite}
\usepackage{amssymb}
\usepackage{amsmath}
\usepackage{amsfonts}
\usepackage{amsthm}
\usepackage{enumerate}
\usepackage{colonequals}
\usepackage{mathrsfs} 
\usepackage{color}
\usepackage{enumitem}
\usepackage{mathtools}
\usepackage[all,cmtip]{xy}
\usepackage{tikz-cd}

\usepackage{url}

\usepackage[mathscr]{euscript}

\newtheorem{theorem}{Theorem}[section]

\newtheorem{proposition}{Proposition}[section]

\newtheorem*{theorem*}{Theorem}
\newtheorem{corollary}{Corollary}[section]
\newtheorem{definition}{Definition}[section]

\newtheorem{conjecture}{Conjecture}[section]

\numberwithin{equation}{section}

\allowdisplaybreaks[2]

\newcommand{\tril}{\triangleleft}
\newcommand{\trir}{\triangleright}

\newcommand{\deee}{\hspace{2 pt} \mathrm{d}}
\newcommand{\rest}{\upharpoonright}

\newcommand{\ov}[1]{\overline{#1}}

\newcommand{\nml}{\left \vert \left \vert}
\newcommand{\nmr}{\right \vert \right \vert}

\renewcommand{\hat}{\widehat}
\renewcommand{\tilde}{\widetilde}

\allowdisplaybreaks

\begin{document}

\title{Extension of positive definite functions and Connes' embedding conjecture }
\author{Peter Burton and Kate Juschenko}

\maketitle

\begin{abstract}
In this paper we formulate a conjecture which is a strengthening of an extension theorem of Bakonyi and Timotin for positive definite functions on the free group on two generators. We prove that this conjecture implies Connes' embedding conjecture. We prove a weak case of this extension conjecture.

\end{abstract}

\tableofcontents

\section{Introduction} \label{part.intro}

\subsection{Connes' embedding conjecture}

\subsubsection{Statement of the conjecture}

If $A$ and $B$ are $C^\ast$-algebras, we will write $A \otimes_{\mathrm{max}} B$ for the maximal tensor product and $A \otimes_{\mathrm{min}} B$ for the minimal tensor product. For information about tensor products of operator algebras we refer the reader to Chapter 11 of \cite{MR1468230}. If $G$ is a countable discrete group, we will write $C^\ast(G)$ for the full group $C^*$-algebra of $G$. For information about group $C^\ast$-algebras we refer the reader to Chapter VII of \cite{MR1402012}. Let $\mathbb{F}$ be the free group on two generators. We now state Connes' embedding conjecture. 

\begin{conjecture}[Connes' embedding conjecture]  \label{thm.CEC} We have \[ C^\ast(\mathbb{F}) \otimes_{\mathrm{max}} C^\ast(\mathbb{F}) = C^\ast(\mathbb{F}) \otimes_{\mathrm{min}} C^\ast(\mathbb{F}) \] \end{conjecture}

The main topic of this paper is a theorem that Conjecture \ref{lem.nogain} in Subsection \ref{sec.ilhanomar} below implies Conjecture \ref{thm.CEC}.\\
\\
Let $\mathbb{F}_\infty$ be the free group on a countably infinite set of generators. It is well known that $\mathbb{F}_\infty$ embeds as a subgroup of $\mathbb{F}$. Given an identification of $\mathbb{F}_\infty$ with a subgroup of $\mathbb{F}$, we obtain an embedding of $C^\ast(\mathbb{F}_\infty)$ into $C^\ast(\mathbb{F})$. This embedding gives rise to a commutative diagram 

\[\begin{tikzcd}        
    C^\ast(\mathbb{F}_\infty) \otimes_{\mathrm{max}} C^\ast(\mathbb{F}_\infty) \arrow{r} \arrow[hook]{d} & C^\ast(\mathbb{F}_\infty) \otimes_{\mathrm{min}} C^\ast(\mathbb{F}_\infty) \arrow[hook]{d} \\ C^\ast(\mathbb{F}) \otimes_{\mathrm{max}} C^\ast(\mathbb{F}) \arrow{r}& C^\ast(\mathbb{F}) \otimes_{\mathrm{min}} C^\ast(\mathbb{F}) \end{tikzcd} \] where all the arrows represent $\ast$-homomorphisms, the vertical arrows are embeddings and Conjecture \ref{thm.CEC} guarantees the bottom arrow is an isomorphism. Thus Conjecture \ref{thm.CEC} implies the following, which is more commonly seen in the literature on Connes' embedding conjecture.
    
\begin{corollary}[Connes' embedding conjecture, $\mathbb{F}_\infty$ version] We have \[ C^\ast(\mathbb{F}_\infty) \otimes_{\mathrm{max}} C^\ast(\mathbb{F}_\infty) = C^\ast(\mathbb{F}_\infty) \otimes_{\mathrm{min}} C^\ast(\mathbb{F}_\infty) \]  \end{corollary}

\subsection{Half finite approximation conjecture} 

\subsubsection{Statement of the conjecture}

We always assume Hilbert spaces are separable with complex scalars. Given a Hilbert space $\mathscr{X}$, let $\mathrm{GL}(\mathscr{X})$ be the group of bounded linear operators on $\mathscr{X}$ with bounded inverses and let $\mathrm{U}(\mathscr{X})$ be the group of unitary operators on $\mathscr{X}$. 

\begin{definition} Let $\mathscr{X}$ be a Hilbert space. We define a linear representation $\zeta:\mathbb{F} \times \mathbb{F} \to \mathrm{GL}(\mathscr{X})$ to be \textbf{half finite} if there exist a finite quotient $\Gamma$ of $\mathbb{F}$ and a linear representation $\zeta_\bullet: \Gamma \times \mathbb{F} \to \mathrm{GL}(\mathscr{X})$ such that $\zeta$ factors as $\zeta_\bullet$ precomposed with $\Pi \times \iota$, where $\Pi:\mathbb{F} \twoheadrightarrow \Gamma$ is the quotient map and $\iota$ is the identity map on $\mathbb{F}$. We define $\zeta$ to be \textbf{totally finite} if there exists a finite quotient $\Lambda$ of $\mathbb{F} \times \mathbb{F}$ and a linear representation $\zeta_\circ:\Lambda \to \mathrm{GL}(\mathscr{X})$ such that $\zeta$ factors as $\zeta_\circ$ precomposed with the quotient map from $\mathbb{F} \times \mathbb{F}$ to $\Lambda$. \end{definition}

Conjecture \ref{thm.CEC} will be an easy consequence of the following.

\begin{conjecture}[Existence of half finite approximations]  \label{thm.half} Let $\mathscr{X}$ be a Hilbert space, let $\rho:\mathbb{F} \times \mathbb{F} \to \mathrm{U}(\mathscr{X})$ be a unitary representation and let $x \in \mathscr{X}$ be a unit vector. Let $E$ and $F$ be finite subsets of $\mathbb{F}$ and let $\epsilon > 0$. Then there exist a Hilbert space $\mathscr{Y}$, a finite quotient $\Gamma$ of $\mathbb{F}$, a half finite unitary representation $\zeta:\mathbb{F} \times \mathbb{F} \to \Gamma \times \mathbb{F} \to \mathrm{U}(\mathscr{Y})$ and a unit vector $y \in \mathscr{Y}$ such that \[ | \langle \rho(g,g') x, x \rangle - \langle \zeta(g,g') y, y \rangle | \leq \epsilon \] for all $g \in E$ and $g' \in F$. \end{conjecture}

\subsubsection{Totally finite approximations}

Conjecture \ref{thm.half} asserts that the half finite representations are dense in the unitary dual of $\mathbb{F} \times \mathbb{F}$. In \cite{MR2077037} it is proved that the representations factoring through finite quotients are dense in the unitary dual of $\mathbb{F}$. Since the tensor product of two such representations is totally finite, given Conjecture \ref{thm.CEC} we can see that the totally finite representations are dense in the unitary dual of $\mathbb{F} \times \mathbb{F}$. In other words, Conjecture \ref{thm.half} implies the same statement with half finite replaced by totally finite.

\subsubsection{Deducing Connes' embedding conjecture from half finite approximation conjecture} \label{subsec.diablo}

In this subsection we will show how Conjecture \ref{thm.half} implies Conjecture \ref{thm.CEC}. Assume that Conjecture \ref{thm.half} is true. Write $||\cdot||_{\mathrm{max}}$ for the norm on $C^\ast(\mathbb{F}) \otimes_{\mathrm{max}} C^\ast(\mathbb{F})$ and $||\cdot||_{\mathrm{min}}$ for the norm on $C^\ast(\mathbb{F}) \otimes_{\mathrm{min}} C^\ast( \mathbb{F})$. Fix an element $\phi$ in the group ring $ \mathbb{C}[\mathbb{F} \times \mathbb{F}]$ such that $||\phi||_{\mathrm{max}} = 1$. In order to prove Conjecture \ref{thm.CEC} suffices to show that $||\phi||_{\mathrm{min}} = 1$. To this end, let $\sigma > 0$.  \\
\\
Since $||\phi||_{\mathrm{max}} = 1$ we can find a Hilbert space $\mathscr{X}$, a unitary representation $\rho: \mathbb{F} \times \mathbb{F} \to \mathrm{U}(\mathscr{X})$ and a unit vector $x \in \mathscr{X}$ such that $||\rho(\phi) x||^2 \geq 1- \sigma$. Write \[ \phi = \sum_{g \in E}\sum_{g' \in F} \alpha_{g,g'}(g,g') \] for finite sets $E,F \subseteq \mathbb{F}$  and complex numbers $(\alpha_{g,g'})_{g \in E,g' \in F}$. Let $\epsilon > 0$ be such that \[ \epsilon \left(\sum_{g \in E} \sum_{g' \in F} |\alpha_{g,g'}| \right)^2 \leq \sigma \] Apply Conjecture \ref{thm.half} to these parameters with $E$ replaced by $E^{-1}E$ and $F$ replaced by $F^{-1}F$. We obtain $\mathscr{Y},\Gamma,\zeta$ and $y$. We have \begin{align*} |  \langle \rho(\phi) x, \rho(\phi) x \rangle &- \langle \zeta(\phi) y, \zeta(\phi) y\rangle | \\ & = \Biggl \vert \left \langle \left( \sum_{g \in E}\sum_{g' \in F} \alpha_{g,g'} \rho(g,g') \right) x, \left( \sum_{h \in E}\sum_{h' \in F} \alpha_{h,h'} \rho(h,h')\right) x \right \rangle \\ & \hspace{1 in} - \left \langle \left( \sum_{g \in E}\sum_{g' \in F} \alpha_{g,g'} \zeta(g,g') \right) y,\left(\sum_{h \in E}\sum_{h' \in F} \alpha_{h,h'} \zeta(h,h')\right) y \right \rangle \Biggr \vert \\ & = \Biggl \vert \sum_{g,h \in E} \sum_{g',h' \in F} \alpha_{g,g'}\ov{\alpha_{h,h'}}  \langle \rho(g,g')x,\rho(h,h') x\ \rangle \\ & \hspace{1 in} - \sum_{g,h \in E} \sum_{g',h' \in F} \alpha_{g,g'}\ov{\alpha_{h,h'}} \bigl \langle \zeta(g,g') y,\zeta(h,h')y \bigr \rangle \Biggr \vert \\ & = \Biggl \vert \sum_{g,h \in E} \sum_{g',h' \in F} \alpha_{g,g'}\ov{\alpha_{h,h'}} \bigl \langle \rho\bigl(h^{-1}g,(h')^{-1}g'\bigr) x, x\bigr \rangle \\ & \hspace{1 in} - \sum_{g,h \in E} \sum_{g',h' \in F} \alpha_{g,g'}\ov{\alpha_{h,h'}} \bigl \langle \zeta\bigl(h^{-1}g,(h')^{-1}g'\bigr) y, y \bigr \rangle \Biggr \vert \\ & \leq \sum_{g,h \in E} \sum_{g',h' \in F} |\alpha_{g,g'}||\alpha_{h,h'}| \bigl \vert \bigl \langle \rho\bigl(h^{-1}g,(h')^{-1}g'\bigr) x, x \bigr \rangle - \bigl \langle \zeta\bigl(h^{-1}g,(h')^{-1}g'\bigr) y, y \bigr \rangle \bigr \vert \\ & \leq \epsilon \sum_{g,h \in E} \sum_{g',h' \in F} |\alpha_{g,g'}||\alpha_{h,h'}| \\   & = \epsilon \left( \sum_{g \in E} \sum_{g' \in F} |\alpha_{g,g'}| \right)^2 \leq \sigma \end{align*}

Therefore we obtain $||\zeta(\phi) y||^2 \geq ||\rho(\phi)x||^2 - \sigma$ so that $||\zeta(\phi)||^2_{\mathrm{op}} \geq 1- 2 \sigma$. \\
\\
There is a natural commutative diagram \[     \xymatrix{ C^\ast(\mathbb{F}) \otimes_{\mathrm{max}} C^\ast(\mathbb{F})  \ar[r] \ar[d] & C^\ast(\mathbb{F}) \otimes_{\mathrm{min}} C^\ast( \mathbb{F}) \ar[d] \\ C^\ast(\Gamma) \otimes_{\mathrm{max}} C^\ast(\mathbb{F}) \ar[r] & C^\ast(\Gamma) \otimes_{\mathrm{min}} C^\ast(\mathbb{F})  } \]  where all the arrows represent surjective $\ast$-homomorphisms. Moreover, there are canonical copies of $\phi$ in each of the above algebras. Since $\zeta$ factors through $\Gamma \times \mathbb{F}$, we see that the norm of $\phi$ in the bottom left corner is at least $\sqrt{1-2\sigma}$. Since $C^\ast(\Gamma)$ is finite dimensional, Lemma 11.3.11 in \cite{MR1468230} implies the arrow across the bottom of the above diagram is an isomorphism. It follows that the norm of $\phi$ in the bottom right corner is at least $\sqrt{1-2\sigma}$ and so $||\phi||_{\mathrm{min}} \geq \sqrt{1-2\sigma}$. Since $\sigma > 0$ was arbitrary we obtain $||\phi||_{\mathrm{min}} = 1$ as required. Therefore in order to prove Conjecture \ref{thm.CEC} it suffices to prove Conjecture \ref{thm.half}.

\subsection{Notation}

\subsubsection{The free group}

Fix a pair of free generators $a$ and $b$ for $\mathbb{F}$ and endow $\mathbb{F}$ with the standard Cayley graph structure corresponding to left multiplication by these generators. If $\Gamma$ is a quotient of $\mathbb{F}$ we identify $a$ and $b$ with their images in $\Gamma$. We write $e$ for the identity of $\mathbb{F}$. We will also use the symbol $e$ for $2.718\ldots$ \\
\\
We consider the word length associated to $a$ and $b$, which we denote by $|\cdot|$. For $r \in \mathbb{N}$ let $\mathbb{B}_r = \{g \in \mathbb{F}:|g| \leq r\}$ be the ball of radius $r$ around $e$. Write $K_r$ for the cardinality of $\mathbb{B}_r$.\\
\\
We define an ordering $\preceq$ on the sphere of radius $1$ in $\mathbb{F}$ by setting $a \preceq b \preceq a^{-1} \preceq b^{-1}$. From this we obtain a corresponding shortlex linear ordering on all of $\mathbb{F}$, which we continue to denote by $\preceq$. For $g \in \mathbb{F}$ define $\mathcal{I}_g = \bigcup \{\{h,h^{-1}\}: h \preceq g\}$. Define a generalized Cayley graph $\mathrm{Cay}(\mathbb{F},g)$ with vertex set equal to $\mathbb{F}$ by placing an edge between distinct elements $h$ and $\ell$ if and only if $\ell^{-1}h \in \mathcal{I}_g$. Write $g_\uparrow$ for the immediate predecessor of $g$ in $\preceq$ and $g_\downarrow$ for the immediate successor of $g$ in $\preceq$. 

\subsubsection{Miscellanea}

If $z$ and $w$ are complex numbers and $\epsilon > 0$ we will sometimes write $z \approx\!\![\epsilon] w$ to mean $|z-w| \leq \epsilon$.\\
\\
We write $\mathbb{D}$ for the open unit disk in the complex plane.\\
\\
If $n \in \mathbb{N}$ we write $[n]$ for $\{1,\ldots,n\}$.

\subsection{Acknowledgements}

We thank Lewis Bowen for several suggestions that improved the writing. We thank Rostyslav Kravchenko for numerous discussions.

\section{Harmonic analysis on the free group} \label{sec.harmonic}

\subsection{The fundamental inequality on the free group}

\begin{definition} Let $d \in \mathbb{N}$ and let $F$ be a finite subset of $\mathbb{F}$. We define a function $\mathsf{C}:F \to \mathrm{Mat}_{d \times d}(\mathbb{C})$ to be \textbf{positive definite} if we have the fundamental inequality \begin{equation} \label{eq.posdef} \sum_{g,h \in E} \alpha(h)^\ast \mathsf{C}(h^{-1}g)\alpha(g) \geq 0 \end{equation} for every subset $E$ of $\mathbb{F}$ with $E^{-1}E \subseteq F$ and every function $\alpha:E \to \mathbb{C}^d$.\\
\\
We define $\mathsf{C}$ to be \textbf{strictly positive definite} if $\mathsf{C}$ is positive definite and the inequality in (\ref{eq.posdef}) is saturated only when $\alpha$ is identically $0$. We define a function $\mathsf{C}:\mathbb{F} \to \mathrm{Mat}_{d \times d}(\mathbb{C})$ to be (strictly) positive definite if $\mathsf{C} \rest F$ is (strictly) positive definite for every finite $F \subseteq \mathbb{F}$. \end{definition}

A positive definite function on the free group can be thought of as a noncommutative analog of an infinite positive definite Toeplitz matrix.

\subsection{The space of normalized strictly positive definite functions}

We will always assume the following normalization condition.

\begin{definition} Let $\mathbf{I}_d$ denote the $d \times d$ identity matrix. If $\mathsf{C}$ is a positive definite function with values in $\mathrm{Mat}_{d \times d}(\mathbb{C})$ whose domain contains $e$, we define $\mathsf{C}$ to be \textbf{normalized} if $\mathsf{C}(e) = \mathbf{I}_d$. \end{definition}

If $\mathsf{C}$ is normalized then for any fixed $g \in \mathbb{F}$ the vectors $\Phi_\mathsf{C}(g)_1,\ldots,\Phi_\mathsf{C}(g)_d$ are orthonormal. We denote the space of normalized strictly positive definite functions $\mathsf{C}:\mathbb{B}_r \to \mathrm{Mat}_{d \times d}(\mathbb{C})$ by $\mathrm{NSPD}(r,d)$. We endow the space of functions from $\mathbb{B}_r$ to $\mathrm{Mat}_{d \times d}(\mathbb{C})$ with the norm \[ ||\mathsf{C}||_1 = \sum_{g \in \mathbb{B}_r} \sum_{j,k=1}^d |\mathsf{C}(g)_{j,k}| \]

\subsection{Realizations of positive definite functions on balls}

Note that $\mathbb{B}_r^{-1}\mathbb{B}_r = \mathbb{B}_{2r}$. Therefore if $\mathsf{C} \in \mathrm{NSPD}(2r,d)$ we can regard it as a positive definite kernel on the set $\mathbb{B}_r \times [d]$. By Theorem C.2.3 in \cite{MR2415834} there exists a Hilbert space $\mathscr{X}(\mathsf{C})$ and a function $\Phi_\mathsf{C}: \mathbb{B}_r \to \mathscr{X}(\mathsf{C})^d$ such that \begin{equation} \label{eq.mari-0} \langle \Phi_\mathsf{C}(g)_j,\Phi_\mathsf{C}(h)_k \rangle = \mathsf{C}(h^{-1}g)_{j,k} \end{equation} for all $g,h \in \mathbb{B}_r$ and all $j,k \in [d]$. Moreover, we may and will assume that the coordinates of the range of $\Phi_\mathsf{C}$ span $\mathscr{X}(\mathsf{C})$. The hypothesis that $\mathsf{C}$ is strictly positive definite ensures that the coordinates of the range of $\Phi_\mathsf{C}$ will be linearly independent. We will refer to them as the canonical basis for $\mathscr{X}(\mathsf{C})$.

\begin{definition} We say that $(\mathscr{X}(\mathsf{C}),\Phi_\mathsf{C})$ as above is a \textbf{realization} of $\mathsf{C}$. \end{definition}

We can construct a realization of a positive definite function $\mathsf{C}:\mathbb{F} \to \mathrm{Mat}_{d \times d}(\mathbb{C})$ in the same way, obtaining a Hilbert space $\mathscr{X}(\mathsf{C})$ and a function $\Phi_\mathsf{C}:\mathbb{F} \to \mathscr{X}(\mathsf{C})^d$ such that the span of the coordinates of the range of $\Phi_\mathsf{C}$ is dense in $\mathscr{X}(\mathsf{C})$.\\
\\
It is clear that given two realizations of the same positive definite function there exists a natural unitary isomorphism from one realization Hilbert space to the other. This isomorphism transforms a canonical basis vector in one realization to the canonical basis vector in another realization having the same index. If $\mathsf{C}:\mathbb{F} \to \mathrm{Mat}_{d \times d}(\mathbb{C})$ is positive definite then the function $g \mapsto (\Phi_\mathsf{C}(hg)_1,\ldots,\Phi_\mathsf{C}(hg)_d)$ is a realization of $\Phi_\mathsf{C}$ for any $h \in \mathbb{F}$. Thus we may make the following definition.

\begin{definition} Let $d \in \mathbb{N}$ and let $\mathsf{C}:\mathbb{F} \to \mathrm{Mat}_{d \times d}(\mathbb{C})$ be positive definite. Then any realization of $\mathsf{C}$ defines an \textbf{associated unitary representation} of $\mathbb{F}$ on $\mathscr{X}(\mathsf{C})$ denoted by $\rho_\mathsf{C}$ and given by the translation $\rho_\mathsf{C}(h)  \Phi_\mathsf{C}(g)_j = \Phi_\mathsf{C}(hg)_j$ for $g,h \in \mathbb{F}$ and $j \in [d]$. \end{definition}

\subsection{Transport operators and relative energies}

\begin{definition} Let $\mathsf{C},\mathsf{D} \in \mathrm{NSPD}(2r,d)$. Let $(\mathscr{X}(\mathsf{C}),\Phi_\mathsf{C})$ and $(\mathscr{X}(\mathsf{D}),\Phi_\mathsf{D})$ be realizations of $\mathsf{C}$ and $\mathsf{D}$ respectively. Define the \textbf{transport operator} $t[\mathsf{C},\mathsf{D}]:\mathscr{X}(\mathsf{C}) \to \mathscr{X}(\mathsf{D})$ by setting \[ t[\mathsf{C},\mathsf{D}]  \sum_{g \in \mathbb{B}_r} \sum_{j=1}^d \alpha(g)_j \Phi_\mathsf{C}(g)_j = \sum_{g \in \mathbb{B}_r} \sum_{j=1}^d \alpha(g)_j \Phi_\mathsf{D}(g)_j \] for functions $\alpha:\mathbb{B}_r \to \mathbb{C}^d$. We refer to the square of the operator norm of $t[\mathsf{C},\mathsf{D}]$ as the \textbf{relative energy} of the pair $(\mathsf{C},\mathsf{D})$ and denote it by $\mathfrak{e}(\mathsf{C},\mathsf{D})$.\\
\\
If $\mathsf{C},\mathsf{D}:\mathbb{F} \to \mathrm{Mat}_{d \times d}(\mathbb{C})$ are strictly positive definite we define the relative energy of the pair $(\mathsf{C},\mathsf{D})$ to be $\sup_{r \in \mathbb{N}} \mathfrak{e}(\mathsf{C} \rest \mathbb{B}_r,\mathsf{D} \rest \mathbb{B}_r)$. We continue to denote it by $\mathfrak{e}(\mathsf{C},\mathsf{D})$. In general we may have $\mathfrak{e}(\mathsf{C},\mathsf{D}) = \infty$. If $\mathfrak{e}(\mathsf{C},\mathsf{D}) < \infty$ then there is a naturally defined transport operator from $\mathscr{X}(\mathsf{C})$ to $\mathscr{X}(\mathsf{D})$, which we continue to denote by $t[\mathsf{C},\mathsf{D}]$. \label{def.transportop} \end{definition}

The relevance of Definition \ref{def.transportop} is that the transport operator between two strictly positive definite functions defined on all of $\mathbb{F}$ clearly intertwines the associated unitary representations. Thus transport operators will be useful in constructing commuting representations of $\mathbb{F}$.

\subsection{Low energy extension conjecture} \label{sec.ilhanomar}

Our main harmonic analysis conjecture is the following. 

\begin{conjecture}[Existence of extensions with low energy gain] \label{lem.nogain} Let $r,d \in \mathbb{N}$ and let $\omega > 0$. Let $\mathsf{C}_1,\ldots,\mathsf{C}_n$ be elements of $\mathrm{NSPD}(r,d)$. Then for each $m \in [n]$ there exists a strictly positive definite function $\hat{\mathsf{C}}_m:\mathbb{F} \to \mathrm{Mat}_{d \times d}(\mathbb{C})$ such that such that $\mathsf{C}_m = \hat{\mathsf{C}}_m \rest \mathbb{B}_r$ and such that $\mathfrak{e}(\mathsf{C}_m,\mathsf{C}_k) = \mathfrak{e}(\hat{\mathsf{C}}_m,\hat{\mathsf{C}}_k)$ for all $m,k \in [n]$. \end{conjecture}

In order for Conjecture \ref{lem.nogain} to be plausible we ought to know that any element of $\mathrm{NSPD}(r,d)$ admits an extension to a positive definite function defined on all of $\mathbb{F}$. This fact appears as Proposition 4.4 in \cite{MR2316876} and as Lemma 25 in \cite{MR3067294}.

\section{Proof of Conjecture \ref{thm.half} from Conjecture \ref{lem.nogain}} \label{sec.aoc}

In Section \ref{sec.aoc} we prove Conjecture \ref{thm.half} from Conjecture \ref{lem.nogain}.

\subsection{Unitary approximate conjugacy of representations}

\subsubsection{Generalities} \label{segment.uac}

We will use the theory of weak containment of unitary representations of countable discrete groups, for which we refer the reader to Appendix H of \cite{MR2583950}. We will say that a unitary representation of a countable discrete group $G$ is maximal if it weakly contains every other unitary representation of $G$. \\
\\
If $G$ is a countable discrete group, $\mathscr{X}$ is a Hilbert space and $\rho: G \to \mathrm{U}(\mathscr{X})$ is a unitary representation, there is a unique extension of $\rho$ to a $\ast$-homomorphism from $C^\ast(G)$ to the algebra $\mathrm{B}(\mathscr{X})$ of bounded operators on $\mathscr{X}$. We denote this extension by $\tilde{\rho}$. Let $\xi: G \to \mathrm{U}(\mathscr{Y})$ be another unitary representation, potentially on a different Hilbert space. By Theorem F.4.4 in \cite{MR2415834}, if $\xi$ is weakly contained in $\rho$ then $||\tilde{\xi}(s)||_{\mathrm{op}} \leq ||\tilde{\rho}(s)||_{\mathrm{op}}$ for all $s \in C^\ast(G)$. It follows that if $\rho$ is a maximal unitary representation then $\tilde{\rho}$ is injective. We now recall a different notion of approximation for representations.

\begin{definition} Unitary representations $\rho:G \to \mathrm{U}(\mathscr{X})$ and $\xi:G \to \mathrm{U}(\mathscr{Y})$ are said to be \textbf{unitarily approximately conjugate} if there is a sequence of unitary operators $u_n:\mathscr{X} \to \mathscr{Y}$ such that for each $g \in G$ we have \[ \lim_{n \to \infty} ||u_n^{-1} \xi(g) u_n - \rho(g)||_{\mathrm{op}} = 0.\] \end{definition}

The following is a special case of Corollary 1.7.5 in \cite{MR2391387}.

\begin{theorem}[Voiculescu] Let $G$ be a countable discrete group. Suppose $\xi$ and $\rho$ are unitary representations of $G$ such that $\tilde{\xi}$ and $\tilde{\rho}$ are injective and such that $\tilde{\xi}(C^*(G))$ and $\tilde{\rho}(C^*(G))$ contain no nonzero compact operators. Then $\xi$ and $\rho$ are unitarily approximately conjugate. \label{thm.voi} \end{theorem}

We can now connect weak containment and unitary approximate conjugacy.

\begin{proposition} Suppose $\xi$ and $\rho$ are maximal unitary representations of $\mathbb{F}$. Then $\xi$ and $\rho$ are unitarily approximately conjugate. \label{prop.uac} \end{proposition}

\begin{proof}[Proof of Proposition \ref{prop.uac}] By Corollary $\mathrm{VII}$.6.7 in \cite{MR1402012} the image of an injective representation of $C^*(\mathbb{F})$ contains no nonzero compact operators. Thus Proposition \ref{prop.uac} follows from Theorem \ref{thm.voi}. \end{proof}

\subsubsection{Introducing initial data} \label{subsec.tau}

Consider the group $G = \mathbb{F} \times \mathbb{F}$. In order to keep a distinction between the factors, we will write $\mathbb{F}_{\triangleleft}$ for the left copy and $\mathbb{F}_{\triangleright}$ for the right copy. We again fix free generators for each copy and endow them with the corresponding word lengths. We will consistently use the letters $g,h$ for elements of $\mathbb{F}_\tril$ and $g',h'$ for elements of $\mathbb{F}_\trir$. If $\mathscr{X}$ is a Hilbert space, $\rho:\mathbb{F}_\tril \times \mathbb{F}_\trir \to \mathrm{GL}(\mathscr{X})$ is a linear representation and $\jmath \in \{\tril,\trir\}$ we will write $\rho_\jmath$ for the restriction of $\rho$ to $\mathbb{F}_\jmath$. We will also write $\mathbb{B}_{r,\jmath}$ for the ball of radius $r$ around the identity in $\mathbb{F}_\jmath$.\\
\\
Fix a Hilbert space $\mathscr{X}$, a unitary representation $\rho:\mathbb{F}_\tril \times \mathbb{F}_\trir \to \mathrm{U}(\mathscr{X})$ and a unit vector $x \in \mathscr{X}$. It is clear that no generality is lost in Conjecture \ref{thm.half} if we assume that $\rho$ is maximal and we indeed make this assumption. Fix finite sets $E \subseteq \mathbb{F}_\tril$, $F \subseteq \mathbb{F}_\trir$ and let $r \in \mathbb{N}$ be such that $E \subseteq \mathbb{B}_{r,\tril}$ and $F \subseteq \mathbb{B}_{r,\trir}$. We may assume that $r \geq 5$. Also fix $\epsilon \in (0,1)$. We have now introduced all the data in the hypotheses of Conjecture \ref{thm.half}.\\
\\
Write $L_{r,\epsilon} = \sqrt{\frac{K_r}{\epsilon}}$. Choose $\delta > 0$ such that \begin{equation} 320L_{r,R}^{10}K_r\delta \leq \epsilon \label{eq.numerical-1} \end{equation}

\subsubsection{Approximate conjugacy with the profinite completion} \label{subsec.conjugacy}

Let $\ov{\mathbb{F}}$ denote the profinite completion of $\mathbb{F}$ and let $\mu$ be its Haar probability measure. For each finite quotient $\Lambda$ of $\mathbb{F}$, there exists a canonical projection $\Pi_\Lambda:\ov{\mathbb{F}} \twoheadrightarrow \Lambda$. Writing $\mathbf{1}_B$ for the indicator function of a subset $B$ of $\ov{\mathbb{F}}$, for each $\lambda \in \Lambda$ we have \begin{align*} \nml\mathbf{1}_{\Pi_{\Lambda}^{-1}(\lambda)}\nmr_2 & = \left( \int_{\ov{\mathbb{F}}} \left \vert \mathbf{1}_{\Pi_\Lambda^{-1}(\lambda)}(\omega) \right \vert^2 \deee \mu(\omega) \right)^{\frac{1}{2}} \\ & = \sqrt{\mu\left(\Pi_\Lambda^{-1}(\lambda)\right)} \\ & = \frac{1}{\sqrt{|\Lambda|}} \end{align*} Moreover, if $\lambda$ and $\lambda'$ are distinct elements of $\Lambda$ then the sets $\Pi_{\Lambda}^{-1}(\lambda)$ and $\Pi_{\Lambda}^{-1}(\lambda')$ are disjoint, so that $\mathbf{1}_{\Pi_\Lambda^{-1}(\lambda)}$ and $\mathbf{1}_{\Pi_\Lambda^{-1}(\lambda')}$ are orthogonal in $L^2(\ov{\mathbb{F}},\mu)$. Therefore the set of functions \[ \left\{ \sqrt{|\Lambda|}\mathbf{1}_{\Pi_{\Lambda^{-1}(\lambda)}}:\lambda \in \Lambda \right\} \] is orthonormal.\\
\\
The profinite structure of $\ov{\mathbb{F}}$ guarantees that the collection of sets \[   \left\{\Pi_\Lambda^{-1}(\lambda):\lambda \in \Lambda,\, \Lambda \mbox{ is a finite quotient of }\mathbb{F} \right\} \] generates the Borel $\sigma$-algebra of $\ov{\mathbb{F}}$. Therefore we have that the span of the functions \begin{equation} \label{eq.circle-1}  \left\{\sqrt{|\Lambda|}\mathbf{1}_{\Pi_\Lambda^{-1}(\lambda)}:\lambda \in \Lambda, \,\Lambda \mbox{ is a finite quotient of }\mathbb{F} \right \} \end{equation} is dense in $L^2(\ov{\mathbb{F}},\mu)$. Choose a sequence $(\Lambda_n)_{n=1}^\infty$ of finite quotients of $\mathbb{F}$ such $\Lambda_n$ is a quotient of $\Lambda_{n+1}$ and such that any finite quotient $\Lambda$ of $\mathbb{F}$ is a quotient of $\Lambda_n$ for some $n \in \mathbb{N}$. Write $\Pi_n$ for $\Pi_{\Lambda_n}$. Then the span of the set of functions \begin{equation} \label{eq.circle-2} \bigcup_{n=1}^\infty\left \{ \sqrt{|\Lambda_n|} \mathbf{1}_{\Pi_n^{-1}(\lambda)}:\lambda \in \Lambda_n \right\} \end{equation} is equal to the span of the set of functions in (\ref{eq.circle-1}). Hence the span of the set of functions in (\ref{eq.circle-2}) is dense in $L^2(\ov{\mathbb{F}},\mu)$. Moreover, the spans of each of the sets inside the union in (\ref{eq.circle-2}) are increasing. \\ 
\\
By considering induced representations, we see that since $\rho$ is a maximal representation of $\mathbb{F}_\tril \times \mathbb{F}_\trir$ we must have that $\rho_\tril$ is a maximal representation of $\mathbb{F}_\tril$.  By Theorem 3.1 in \cite{MR2931406}, the left translation action of $\mathbb{F}$ on $(\ov{\mathbb{F}},\mu)$ is maximal in the order of weak containment among measure preserving actions of $\mathbb{F}$. We refer the reader to Chapter 10 of \cite{MR2583950} for information on this variant of weak containment, but all we will need to know about it is that Proposition 10.5 and Theorem E.1 in \cite{MR2583950} imply that the Koopman representation of a maximal action is a maximal representation. Write $\ov{\kappa}:\mathbb{F} \to \mathrm{U}(L^2(\ov{\mathbb{F}},\mu))$ for the Koopman representation of the left translation action, so that Proposition \ref{prop.uac} implies $\rho_\tril$ and $\ov{\kappa}$ are unitarily approximately conjugate. Let $u:\mathscr{X} \to L^2(\ov{\mathbb{F}},\mu)$ be a unitary operator such that \[||u^{-1} \ov{\kappa}(g)u - \rho_\tril(g)||_{\mathrm{op}}  \leq \delta \] for all $g \in \mathbb{B}_{1,\tril}$. \\
\\
Consider the vector $u x$. Our previous discussion of (\ref{eq.circle-2}) implies that we can find $n \in \mathbb{N}$ and a function $\alpha: \Lambda_n \to \mathbb{C}$ with\[ \nml u x - \sum_{\lambda \in \Lambda_n} \alpha(\lambda) \sqrt{|\Lambda_n|}\mathbf{1}_{\Pi_n^{-1}(\lambda)} \nmr_2 \leq \delta \]  We may assume $n$ is large enough that the balls of radius $4R$ in $\mathbb{F}/\Lambda_n$ are isomorphic to the balls of radius $4R$ in $\mathbb{F}$.\\
\\
The partition $\{\Pi_n^{-1}(\lambda):\lambda \in \Lambda_n\}$ of $\ov{\mathbb{F}}$ is permuted by the left translation action of $\mathbb{F}$ on $\ov{\mathbb{F}}$ so that $g\Pi_n^{-1}(\lambda) = \Pi_n^{-1}(g\lambda)$ for all $g \in \mathbb{F}$ and $\lambda \in \Lambda_n$. Thus we have $\kappa(g)  \mathbf{1}_{\Pi_n^{-1}(\lambda)} = \mathbf{1}_{\Pi_n^{-1}(g\lambda)}$. Since the vectors \[\left \{ \sqrt{|\Lambda_n|} \mathbf{1}_{\Pi_n^{-1}(\lambda)}: \lambda \in \Lambda_n \right\} \] are orthonormal and $x$ is a unit vector, we may assume that \[ \sum_{\lambda \in \Lambda_n} |\alpha(\lambda)|^2 = 1 \]  Enumerate $\Lambda_n = \{\lambda_1,\ldots,\lambda_d\}$ and for $j \in [d]$ let \[ x_j = u^{-1}  \sqrt{d} \,\mathbf{1}_{\Pi^{-1}_n(\lambda_j)} \] We may assume without loss of generality that $d \geq R$. Write $\kappa = u^{-1} \ov{\kappa} u$. We summarize the objects we have just constructed.

\begin{itemize} \item An orthonormal set of vectors $x_1,\ldots,x_d$ in $\mathscr{X}$ and an element $\alpha \in \mathbb{C}^d$ with \begin{equation} \sum_{j=1}^d |\alpha_j|^2 = 1 \label{eq.alphal2} \end{equation} such that \begin{equation} ||x - \alpha_1 x_1 - \cdots - \alpha_d x_d|| \leq \delta \label{eq.uac-x} \end{equation}   \item An action $\sigma:\mathbb{F}_\tril \to \mathrm{Sym}(d)$. \item A unitary representation $\kappa:\mathbb{F}_\tril \to \mathrm{U}(\mathscr{X})$ such that \begin{equation} ||\kappa(g) - \rho_\tril(g)||_{\mathrm{op}} \leq \delta \label{eq.uac-g} \end{equation} for all $g \in \mathbb{B}_{1,\tril}$ and such that \begin{equation} \kappa(g) x_j = x_{\sigma(g)j}  \label{eq.uac-k} \end{equation} for all $g \in \mathbb{F}_\tril$ and all $j \in [d]$. \end{itemize}

\subsection{Building a half finite linear representation}

\subsubsection{Proximity between inner products at individual nodes}

\begin{proposition} \label{prop.prox-ind} Let $g',h' \in \mathbb{B}_{r,\trir}$ and let $\beta,\eta \in \mathbb{C}^d$. Also let $g \in \mathbb{B}_{1,\tril}$ and let $\varsigma \in \mathrm{Sym}(d)$. We have \begin{align*} \Biggl \vert \Biggl \langle  \rho_\trir(g')  \sum_{j=1}^d \beta_j x_{\varsigma j},  \rho_\trir(h')  \sum_{k=1}^d \eta_k x_{\varsigma k} \Biggr \rangle   - & \Biggl \langle \rho_\trir(g')   \sum_{j=1}^d \beta_jx_{\sigma(g)\varsigma j},  \rho_\trir(h')  \sum_{k=1}^d \eta_k x_{\sigma(g)\varsigma k} \Biggr \rangle \Biggr \vert \\ & \leq 2\delta ||\beta||_2||\eta||_2  \end{align*} \end{proposition}

\begin{proof}[Proof of Proposition \ref{prop.prox-ind}] All norms and inner products in the proof of Proposition \ref{prop.prox-ind} will be in $\mathscr{X}$. From (\ref{eq.uac-k}) we have \begin{align} & \Biggl \vert \Biggl \langle  \rho_\trir(g')  \sum_{j=1}^d \beta_j x_{\varsigma j},  \rho_\trir(h')  \sum_{k=1}^d \eta_k x_{\varsigma k} \Biggr \rangle   -  \left \langle \rho_\trir(g')   \sum_{j=1}^d \beta_jx_{\sigma(g)\varsigma j},  \rho_\trir(h')  \sum_{k=1}^d \eta_k x_{\sigma(g)\varsigma k} \right \rangle \Biggr \vert \nonumber \\ & = \Biggl \vert \left \langle \rho_\trir(g')  \sum_{j=1}^d \beta_jx_{\varsigma j}, \rho_\trir(h')  \sum_{k=1}^d \eta_k x_{\varsigma k} \right \rangle  -  \left \langle \rho_\trir(g')\kappa(g)  \sum_{j=1}^d \beta_j x_{\varsigma j}, \rho_\trir(h') \kappa(g)   \sum_{k=1}^d \eta_k x_{\varsigma k}  \right \rangle \Biggr \vert  \label{eq.obtain-2} \end{align}

Write $x_\beta$ for $\sum_{j=1}^d \beta_j x_{\varsigma j}$ and $x_\eta$ for $\sum_{k=1}^d \eta_k x_{\varsigma k}$. We compute

\begin{align} (\ref{eq.obtain-2}) &\leq \bigl \vert \bigl \langle \rho_\trir(g') x_\beta, \rho_\trir(h') x_\eta \bigr \rangle  -  \bigl \langle  \rho_\trir(g')\rho_\tril(g) x_\beta, \rho_\trir(h') \rho_\tril(g)  x_\eta \bigr \rangle \bigr \vert \label{eq.obtain-2.5} \\ & \hspace{0.5 in} + \bigl \vert \bigl \langle \rho_\trir(g')\kappa(g)  x_\beta,  \rho_\trir(h')\kappa(g) x_\eta \bigr \rangle -  \bigl \langle \rho_\trir(g')\rho_\tril(g) x_\beta, \rho_\trir(h') \rho_\tril(g) x_\eta \bigr \rangle \bigr \vert  \nonumber \\ & = \bigl \vert \bigl \langle \rho_\trir(g') x_\beta, \rho_\trir(h') x_\eta \bigr \rangle  -  \bigl \langle \rho_\tril(g)  \rho_\trir(g') x_\beta, \rho_\tril(g) \rho_\trir(h')  x_\eta \bigr \rangle \bigr \vert \label{eq.obtain-4} \\ & \hspace{1 in} + \bigl \vert \bigl \langle \rho_\trir(g')\kappa(g)  x_\beta,  \rho_\trir(h')\kappa(g) x_\eta \bigr \rangle  -  \bigl \langle \rho_\trir(g')\rho_\tril(g) x_\beta, \rho_\trir(h') \rho_\tril(g) x_\eta \bigr \rangle \bigr \vert \label{eq.obtain-5}   \\ & = \bigl \vert \bigl \langle \rho_\trir(g')\kappa(g)  x_\beta,  \rho_\trir(h')\kappa(g) x_\eta \bigr \rangle  -  \bigl \langle \rho_\trir(g')\rho_\tril(g) x_\beta, \rho_\trir(h') \rho_\tril(g) x_\eta \bigr \rangle \big \vert \label{eq.obtain-6}  \\ & \leq \bigl \vert \bigl \langle \rho_\trir(g')\kappa(g)  x_\beta,  \rho_\trir(h')\kappa(g) x_\eta \bigr \rangle  - \bigl \langle \rho_\trir(g')\kappa(g) x_\beta, \rho_\trir(h') \rho_\tril(g) x_\eta \bigr \rangle \bigr \vert \nonumber  \\ & \hspace{1 in} + \bigl \vert  \bigl \langle \rho_\trir(g')\kappa(g) x_\beta, \rho_\trir(h') \rho_\tril(g) x_\eta \bigr \rangle -  \bigl \langle \rho_\trir(g')\rho_\tril(g) x_\beta, \rho_\trir(h') \rho_\tril(g) x_\eta \bigr \rangle \big \vert  \nonumber \\ & = \bigl \vert \bigl \langle \rho_\trir(g')\kappa(g) x_\beta, \rho_\trir(h')(\kappa(g) - \rho_\tril(g))    x_\eta \bigr \rangle \bigr \vert \nonumber  \\ & \hspace{1 in} + \bigl \vert \bigl \langle \rho_\trir(g')(\kappa(g)-\rho_\tril(g)) x_\beta, \rho_\trir(h')\rho_\tril(g)x_\eta \bigr \rangle \bigr \vert  \nonumber  \\ & \leq || \rho_\trir(g')\kappa(g) x_\beta|| \, || \rho_\trir(h')(\kappa(g) - \rho_\tril(g))    x_\eta || \label{eq.obtain-6.95} \\ & \hspace{1 in} +||\rho_\trir(g')(\kappa(g) -\rho_\tril(g)) x_\beta|| \, || \rho_\trir(h')\rho_\tril(g) x_\eta||  \label{eq.obtain-7}   \\ & = || x_\beta || \, || (\kappa(g) - \rho_\tril(g))    x_\eta || + || (\kappa(g) -\rho_\tril(g)) x_\beta|| \, || x_\eta|| \label{eq.obtain-7.5} \\ &\leq 2\, ||\kappa(g) -\rho_\tril(g)||_{\mathrm{op}} || x_\beta|| \, || x_\eta||  \label{eq.obtain-9}   \\ & = 2\, ||\kappa(g) -\rho_\tril(g)||_{\mathrm{op}} ||\beta||_2 ||\eta||_2   \label{eq.obtain-10} \\ & \leq  2\delta ||\beta||_2||\eta||_2 \label{eq.obtain-11} \end{align}

Here, \begin{itemize} \item (\ref{eq.obtain-2.5}) is equal to (\ref{eq.obtain-4}) since $\rho_\tril$ and $\rho_\trir$ commute, \item (\ref{eq.obtain-6}) follows from (\ref{eq.obtain-4}) - (\ref{eq.obtain-5}) since $\rho_\tril$ is unitary and therefore (\ref{eq.obtain-4}) is $0$, \item (\ref{eq.obtain-7.5}) follows from (\ref{eq.obtain-6.95}) - (\ref{eq.obtain-7}) since $\rho_\trir$ and $\kappa$ are unitary, \item (\ref{eq.obtain-10}) follows from (\ref{eq.obtain-9}) since $x_1,\ldots,x_d$ is orthonormal, \item and (\ref{eq.obtain-11}) follows from (\ref{eq.obtain-10}) by (\ref{eq.uac-g}) \end{itemize} Proposition \ref{prop.prox-ind} follows by combining (\ref{eq.obtain-2}) with (\ref{eq.obtain-11}).  \end{proof}

\subsubsection{Constructing a family of permuted positive definite functions}

For $\varsigma \in \mathrm{Sym}(d)$ define a positive definite function $\mathsf{C}_\varsigma:\mathbb{B}_{2r,\trir} \to \mathrm{Mat}_{d \times d}(\mathbb{C})$ by setting \begin{equation} \mathsf{C}_\varsigma\bigl((h')^{-1} g'\bigr)_{j,k} = \langle \rho_\trir(g') x_{\varsigma j}, \rho_\trir(h') x_{\varsigma k} \rangle \label{eq.1-1} \end{equation} for $g',h' \in \mathbb{B}_{r,\trir}$. Also define $\Delta_r:\mathbb{B}_{2r,\trir} \to \mathrm{Mat}_{d \times d}(\mathbb{C})$ by setting \[ \Delta_r(g') = \begin{cases} \mathbf{I}_d & \mbox{ if }g' =e \\ \mathbf{0}_d & \mbox{ if } g' \in \mathbb{B}_{2r} \setminus \{e\}  \end{cases} \] where $\mathbf{0}_d$ denotes the $d \times d$ zero matrix. Let $\mathsf{D}_\varsigma = (1-\epsilon )\mathsf{C}_\varsigma + \epsilon \Delta_r$. Note that for any function $\beta:\mathbb{B}_{r,\trir} \to \mathbb{C}^d$ we have \begin{align} & \nml \sum_{g' \in \mathbb{B}_{r,\trir}} \sum_{j=1}^d \beta(g')_j \Phi_{\mathsf{D}_\varsigma}(g')_j \nmr^2 \nonumber \\ & \hspace{0.5 in} =  (1-\epsilon)\nml \sum_{g' \in \mathbb{B}_{r,\trir}} \sum_{j=1}^d \beta(g')_j \Phi_{\mathsf{C}_\varsigma}(g')_j \nmr^2 + \epsilon \nml \sum_{g' \in \mathbb{B}_{r,\trir}} \sum_{j=1}^d \beta(g')_j \Phi_{\Delta_r}(g')_j \nmr^2 \nonumber \\ & \hspace{0.5 in} \geq \epsilon \nml \sum_{g' \in \mathbb{B}_{r,\trir}} \sum_{j=1}^d \beta(g')_j \Phi_{\Delta_r}(g')_j \nmr^2 \label{eq.hungry-1} \\ &\hspace{0.5 in} = \epsilon\left( \sum_{g' \in \mathbb{B}_{r,\trir}} \sum_{j=1}^d |\beta(g')_j|^2 \right) \label{eq.ell2} \end{align} Here, (\ref{eq.ell2}) follows from (\ref{eq.hungry-1}) since the set \[ \{\Phi_{\Delta_r}(g')_j:g' \in \mathbb{B}_{r,\trir},j \in [d]\} \] is orthonormal.

\subsubsection{Establishing bounds on transport operators}

\begin{proposition} \label{prop.1-1} We have $\mathfrak{e}(\mathsf{D}_\varsigma,\mathsf{D}_\varrho) \leq L_{r,\epsilon}$ for all $\varsigma,\varrho \in \mathrm{Sym}(d)$. \end{proposition}

\begin{proof}[Proof of Proposition \ref{prop.1-1}] Let $\beta: \mathbb{B}_{r,\trir} \to \mathbb{C}^d$ be such that \[  y = \sum_{g' \in \mathbb{B}_{r,\trir}} \sum_{j=1}^d \beta(g')_j \Phi_{\mathsf{D}_\varsigma}(g')_j \] is a unit vector in $\mathscr{X}(\mathsf{D}_\varsigma)$. Thus from (\ref{eq.ell2}) we have \begin{equation} 1 \geq \epsilon \left( \sum_{g' \in \mathbb{B}_{r,\trir}} \sum_{j=1}^d|\beta(g')_j|^2 \right) \label{eq.money-55} \end{equation}

We compute \begin{align} || t[\mathsf{D}_\varsigma,\mathsf{D}_\varrho] y|| & = \nml  \sum_{g' \in \mathbb{B}_{r,\trir}} \sum_{j=1}^d \beta(g')_j \Phi_{\mathsf{D}_\varrho}(g')_j \nmr \nonumber \\ & \leq \sum_{g' \in \mathbb{B}_{r,\trir}} \nml \sum_{j=1}^d \beta(g')_j \Phi_{\mathsf{D}_\varrho}(g')_j \nmr \label{eq.2-2} \\ & = \sum_{g' \in \mathbb{B}_{r,\trir}} \left( \sum_{j=1}^d |\beta(g')_j|^2 \right)^{\frac{1}{2}} \label{eq.2-3} \\ & \leq \sqrt{K_r} \left( \sum_{g' \in \mathbb{B}_{r,\trir}} \sum_{j=1}^d|\beta(g')_j|^2 \right)^{\frac{1}{2}}  \label{eq.22-1} \\ & \leq \sqrt{\frac{K_r}{\epsilon}} \label{eq.22-2} \end{align} Here, (\ref{eq.2-3}) follows from (\ref{eq.2-2}) since $\mathsf{D}_\varrho$ is normalized and (\ref{eq.22-2}) follows from (\ref{eq.22-1}) by (\ref{eq.money-55}). \end{proof}

\begin{proposition} We have $\mathfrak{e}(\mathsf{D}_{\varsigma},\mathsf{D}_{\sigma(g)\varsigma}) \leq 1+2K_rR\delta$ for all $\varsigma \in \mathrm{Sym}(d)$ and all $g \in \mathbb{B}_{1,\tril}$. \label{prop.eggs} \end{proposition}

\begin{proof}[Proof of Proposition \ref{prop.eggs}] Let $\varsigma \in \mathrm{Sym}(d)$ and $g \in \mathbb{B}_{1,\tril}$ Let $\beta: \mathbb{B}_{r,\trir} \to \mathbb{C}^d$ be such that if we write \[ y = \sum_{g' \in \mathbb{B}_{r,\trir}} \sum_{j=1}^d \beta(g')_j \Phi_{\mathsf{D}_\varsigma}(g')_j \] then $y$ is a unit vector in $\mathscr{X}(\mathsf{D}_\varsigma)$. Thus from (\ref{eq.ell2}) we have  \begin{equation} 1 \geq \frac{1}{R} \left( \sum_{g' \in \mathbb{B}_{r,\trir}} \sum_{j=1}^d|\beta(g')_j|^2 \right) \label{eq.money-56} \end{equation}
We compute 

\begin{align} \bigl \vert ||t[\mathsf{D}_\varsigma&,\mathsf{D}_{\sigma(g)\varsigma}] y||^2 -1 \bigr \vert \nonumber \\ & = \Biggl \vert  \left \langle \sum_{g' \in \mathbb{B}_{r,\trir}} \sum_{j=1}^d \beta(g')_j \Phi_{\mathsf{D}_{\sigma(g)\varsigma}}(g')_j, \sum_{h' \in \mathbb{B}_{r,\trir}} \sum_{k=1}^d \beta(h')_k \Phi_{\mathsf{D}_{\sigma(g) \varsigma}}(h')_k \right \rangle \nonumber \\ & \hspace{1 in} -  \left \langle \sum_{g' \in \mathbb{B}_{r,\trir}} \sum_{j=1}^d \beta(g')_j \Phi_{\mathsf{D}_{\varsigma}}(g')_j, \sum_{h' \in \mathbb{B}_{r,\trir}} \sum_{k=1}^d \beta(h')_k \Phi_{\mathsf{D}_{\varsigma}}(h')_k \right \rangle \Biggr \vert \label{eq.mari-1} \\ & =  \Biggl \vert  \sum_{g',h' \in \mathbb{B}_{r,\trir}}  \sum_{j,k=1}^d \beta(g')_j \ov{\beta(h')_k}\mathsf{D}_{\sigma(g) \varsigma} \bigl( (h')^{-1}g'\bigr)_{j,k} \nonumber \\ & \hspace{1 in} - \sum_{g',h' \in \mathbb{B}_{r,\trir}}  \sum_{j,k=1}^d \beta(g')_j \ov{\beta(h')_k}  \mathsf{D}_{\varsigma}\bigl((h')^{-1}g'\bigr)_{j,k} \Biggr \vert \label{eq.mari-2} \\ & \leq \sum_{g',h' \in \mathbb{B}_{r,\trir}}   \left \vert \sum_{j,k=1}^d \beta(g')_j \ov{\beta(h')_k} \left( \mathsf{D}_{\sigma(g) \varsigma} \bigl( (h')^{-1}g'\bigr)_{j,k} - \mathsf{D}_{\varsigma}\bigl((h')^{-1}g'\bigr)_{j,k} \right) \right \vert \label{eq.3-1}  \\ & = \left(1-\frac{1}{R}\right)\sum_{g',h' \in \mathbb{B}_{r,\trir}}  \left \vert \sum_{j,k=1}^d \beta(g')_j \ov{\beta(h')_k} \left( \mathsf{C}_{\sigma(g) \varsigma} \bigl( (h')^{-1}g'\bigr)_{j,k} - \mathsf{C}_{\varsigma}\bigl((h')^{-1}g'\bigr)_{j,k} \right) \right \vert \label{eq.3-2} \\ & \leq \sum_{g',h' \in \mathbb{B}_{r,\trir}}  \left \vert \sum_{j,k=1}^d \beta(g')_j \ov{\beta(h')_k} \left( \mathsf{C}_{\sigma(g) \varsigma} \bigl( (h')^{-1}g'\bigr)_{j,k} - \mathsf{C}_{\varsigma}\bigl((h')^{-1}g'\bigr)_{j,k} \right) \right \vert \label{eq.3-87}  \\ & = \sum_{g',h' \in \mathbb{B}_{r,\trir}} \biggl \vert \sum_{j,k=1}^d \beta(g')_j \ov{\beta(h')_k} \bigl(\langle \rho_\trir(g') x_{\sigma(g)\varsigma j},\rho_\trir(h') x_{\sigma(g) \varsigma k} \rangle -\langle \rho_\trir(g') x_{\varsigma j},\rho_\trir(h') x_{\varsigma k} \rangle \bigr) \biggr \vert \label{eq.3-3} \\ & = \sum_{g',h' \in \mathbb{B}_{r,\trir}} \Biggl \vert  \Biggl( \left \langle \rho_\trir(g')   \sum_{j=1}^d \beta(g')_j x_{\sigma(g)\varsigma j} , \rho_\trir(h')  \sum_{k=1}^d \beta(h')_k x_{\sigma(g) \varsigma k} \right \rangle \nonumber \\ & \hspace{1 in} -  \left \langle \rho_\trir(g')   \sum_{j=1}^d \beta(g')_j x_{\varsigma j} , \rho_\trir(h')  \sum_{k=1}^d \beta(h')_k x_{\varsigma k} \right \rangle  \Biggr) \Biggr \vert \label{eq.3-4}  \\ & \leq 2\delta \sum_{g',h' \in \mathbb{B}_{r,\trir}} \left( \sum_{j=1}^d |\beta(g')_j|^2 \right)^\frac{1}{2} \left( \sum_{k=1}^d |\beta(h')_k|^2 \right) ^{\frac{1}{2}} \label{eq.3-5} \\ & \leq 2K_r\delta \left( \sum_{g' \in \mathbb{B}_{r,\trir}} \sum_{j=1}^d |\beta(g')_j|^2 \right) \label{eq.10-1}  \\ & \leq 2K_rR\delta \label{eq.10-2} \end{align}

Here, \begin{itemize} \item (\ref{eq.mari-2}) follows from (\ref{eq.mari-1}) using (\ref{eq.mari-0}) \item (\ref{eq.3-2}) follows from (\ref{eq.3-1}) since the $\Delta_r$ components in the definitions of $\mathsf{D}_\varsigma$ and $\mathsf{D}_{\sigma(g) \varsigma}$ cancel, \item (\ref{eq.3-3}) follows from (\ref{eq.3-87}) by (\ref{eq.1-1}) \item  (\ref{eq.3-5}) follows from (\ref{eq.3-4}) by Proposition \ref{prop.prox-ind} \item and (\ref{eq.10-2}) follows from (\ref{eq.10-1}) by (\ref{eq.money-56}). \end{itemize} 

\end{proof}

 \subsubsection{Constructing a representation through permutations}

Apply Conjecture \ref{lem.nogain} to the positive definite functions $(\mathsf{D}_v)_{v \in V}$  to obtain positive definite functions $(\hat{\mathsf{D}}_v)_{v \in V}$ such that $\mathsf{D}_v = \hat{\mathsf{D}}_v \rest \mathbb{B}_r$ and such that $\mathfrak{e}(\mathsf{D}_v,\mathsf{D}_w) = \mathfrak{e}(\hat{\mathsf{D}}_v,\hat{\mathsf{D}}_w)$ for all $v,w \in V$. We regard each $\hat{\mathsf{D}}_v$ as a function from $\mathbb{F}_\trir$ to $\mathrm{Mat}_{d \times d}(\mathbb{C})$. \\
\\
Write $\mathscr{Y}_v$ for $\mathscr{X}(\mathsf{D}_v)$ and let $\zeta_{\trir,v}:\mathbb{F}_\trir \to \mathrm{U}(\mathscr{Y}_v)$ be the associated representation of $\mathsf{D}_v$. Define $\mathscr{Y} = \bigoplus_{v \in V} \mathscr{Y}_v$ and $\zeta_\trir = \bigoplus_{v \in V} \zeta_{\trir,v}$. Define a representation $\theta:\mathbb{F}_\tril \to \mathrm{GL}(\mathscr{Y})$ by setting \[ \theta(g)  = \bigoplus_{v \in V} t[\mathsf{D}_v,\mathsf{D}_{\sigma(g)v}] \] Note that $\theta$ factors through the finite group $\Gamma = \sigma(\mathbb{F}_\tril)$. Moreover, we have \[ t[\mathsf{D}_v,\mathsf{D}_{\sigma(g)v}]\zeta_{\trir.v} = \zeta_{\trir,\tau(g)v}t[\mathsf{D}_v,\mathsf{D}_{\sigma(g)v}] \] for all $v \in V$ and $g \in \mathbb{F}_\tril$. Therefore $\theta$ commutes with $\zeta_\trir$ so that $\theta \times \zeta_\trir$ is a half finite linear representation of $G$. \\
\\
From Proposition \ref{prop.1-1} we see \begin{equation} ||\theta(g)|| \leq L_{r,\epsilon} \label{eq.you} \end{equation} for all $g \in \mathbb{F}_\tril$ and from Proposition \ref{prop.eggs} we see that \begin{equation} \label{eq.becky-1} ||\theta(g)||_{\mathrm{op}} \leq 1+2K_r\delta \end{equation} for all $g \in \mathbb{B}_{r,\tril}$

\subsection{Repairing the representation to be unitary}

\subsubsection{Conjugation by an average} \label{segment.repair0}

In Segments \ref{segment.repair0} and \ref{segment.spectrum} we regard $\theta$ as a representation of the finite group $\Gamma$. Define a positive operator $q \in \mathrm{B}(\mathscr{Y})$ by \[ q= \frac{1}{|\Gamma|} \sum_{\gamma \in \Gamma} \theta(\gamma)^\ast \theta(\gamma). \] 

By applying (\ref{eq.becky-3}) to $g^{-1}$ we see that each $\theta(\gamma)$ is invertible. Hence each operator $\theta(\gamma)^\ast\theta(\gamma)$ is strictly positive and so $q$ is invertible. Define a representation $\zeta_\tril$ of $\Gamma$ on $\mathscr{Y}$ by setting $\zeta_\tril(\gamma) = q^\frac{1}{2} \theta(\gamma) q^{-\frac{1}{2}}$. For all $\gamma \in \Gamma$ and all $g' \in \mathbb{F}_\trir$ we have that $\zeta_\trir(g')$ commutes with each $\theta(\gamma)$. Since $\zeta_\trir(g')$ is unitary, this implies that $\zeta_\trir(g')$ commutes with $\theta(\gamma)^\ast$ and hence $\zeta_\trir(g')$ commutes with $q$. Therefore $\zeta_\trir$ commutes with $\zeta_\tril$ and so if we set $\zeta = \zeta_\tril \times \zeta_\trir$ then $\zeta$ is a half finite linear representation of $G$. \\
\\
We claim that $\zeta$ is in fact unitary. Write $I$ for the identity operator on $\mathscr{Y}$. For $\gamma \in \Gamma$ we have \begin{align} \zeta_\tril(\gamma)^\ast \zeta_\tril(\gamma) &= \bigl( q^{\frac{1}{2}} \theta(\gamma) q^{-\frac{1}{2}}\bigr)^\ast \bigl( q^{\frac{1}{2}} \theta(\gamma) q^{-\frac{1}{2}} \bigr)  \nonumber \\ & = q^{-\frac{1}{2}} \theta(\gamma)^\ast q \theta(\gamma) q^{-\frac{1}{2}}  \nonumber \\ & = q^{-\frac{1}{2}} \theta(\gamma)^\ast \left( \frac{1}{|\Gamma|} \sum_{\nu \in \Gamma} \theta(\nu)^\ast \theta(\nu) \right)   \theta(\gamma) q^{-\frac{1}{2}}  \nonumber \\ & = q^{-\frac{1}{2}} \left( \frac{1}{|\Gamma|} \sum_{\nu \in \Gamma} \theta(\gamma)^\ast \theta(\nu)^\ast \theta(\nu) \theta(\gamma) \right) q^{-\frac{1}{2}} \nonumber \\ & = q^{-\frac{1}{2}} \left( \frac{1}{|\Gamma|} \sum_{\nu \in \Gamma} \theta(\nu \gamma)^\ast \theta(\nu \gamma) \right) q^{-\frac{1}{2}}  \nonumber \\ & = q^{-\frac{1}{2}} \left(\frac{1}{|\Gamma|} \sum_{\nu \in \Gamma} \theta(\nu)^\ast \theta(\nu) \right) q^{-\frac{1}{2}} \nonumber \\ & = I \label{eq.6-1} \end{align} so that $\zeta_\tril(\gamma)$ is unitary and therefore $\zeta$ is a unitary representation. 

\subsubsection{Bounding the spectrum of the average} \label{segment.spectrum}

\begin{proposition} We have $\mathrm{spec}(q) \subseteq [L_{r,R}^{-2},L_{r,R}^2]$. \label{prop.spec} \end{proposition}

\begin{proof}[Proof of Proposition \ref{prop.spec}] Using (\ref{eq.becky-3}) we see that for any unit vector $y \in \mathscr{Y}$ we have \begin{align} \langle q y,y \rangle & = \frac{1}{|\Gamma|} \sum_{\gamma \in \Gamma} \bigl \langle \theta(\gamma)^\ast \theta(\gamma) y,y \bigr \rangle \nonumber \\ & = \frac{1}{|\Gamma|} \sum_{\gamma \in \Gamma} ||\theta(\gamma) y||^2 \nonumber \\ & \leq L_{r,R}^2 \label{eq.repair1} \end{align} By applying (\ref{eq.becky-3}) to $g^{-1}$ we see that \[ \inf \bigl \{ ||\theta(\gamma) y||^2 : y \in \mathscr{Y} \mbox{ is a unit vector} \bigr \} \geq \frac{1}{L_{r,R}^2} \] and so \begin{equation} \label{eq.repair2}  \inf \bigl \{ \langle q y,y \rangle :y \in \mathscr{Y} \mbox{ is a unit vector} \bigr \} \geq \frac{1}{L_{r,R}^2}. \end{equation}  Now suppose $\lambda \in \mathrm{spec}(q)$. Since $q$ is self-adjoint, there exists a sequence $(y_n)_{n=1}^\infty$ of unit vectors in $\mathscr{Y}$ such that $\lim_{n \to \infty} ||(q -  \lambda I) y_n|| = 0$. This implies that  $\lim_{n \to \infty}  \langle (q-\lambda I) y_n,y_n \rangle = 0$ and so $\lim_{n \to \infty} \langle q y_n,y_n \rangle = \lambda$. Thus from (\ref{eq.repair1}) and (\ref{eq.repair2}) we have $L_{r,R}^{-2} \leq \lambda \leq L_{r,R}^2$. \end{proof} 

\subsubsection{Estimating the distance to the repaired representation} \label{segment.repair}

\begin{proposition} Suppose $g \in \mathbb{B}_{r,\tril}$. Then $||\zeta_\tril(g) - \theta(g)||_{\mathrm{op}} \leq R^{-1}$. \label{prop.repair1} \end{proposition}

\begin{proof}[Proof of Proposition \ref{prop.repair1}] Fix $g \in \mathbb{B}_{r,\tril}$. By applying (\ref{eq.becky-1}) to $g$ and $g^{-1}$ we see \[ \frac{1}{1+2K_r\delta}I \leq \theta(g)^\ast \theta(g) \leq  (1+2K_r\delta)I \] Since $\theta(g)^\ast\theta(g)$ is unitarily conjugate to $\theta(g)\theta(g)^\ast$ we obtain\[  \frac{1}{1+2K_r\delta}I \leq \theta(g) \theta(g)^\ast \leq  (1+2K_r\delta)I \] so that \begin{equation} ||\theta(g) \theta(g)^\ast - I ||_{\mathrm{op}} \leq 2K_r\delta\label{eq.becky-13} \end{equation} Since $q^{-\frac{1}{2}}\theta(g)^\ast q \theta(g) q^{-\frac{1}{2}} = I$ we have $\theta(g)^\ast q \theta(g) = q$. Therefore \begin{align} ||q \theta(g) - \theta(g)q||_{\mathrm{op}}& = ||q\theta(g) -  \theta(g)\theta(g)^\ast q \theta(g)||_{\mathrm{op}} \nonumber  \\ & \leq || I - \theta(g) \theta(g)^\ast||_{\mathrm{op}} ||q||_{\mathrm{op}} ||\theta(g)||_{\mathrm{op}} \label{eq.repair-1-1} \\ &  \leq 2 || I - \theta(g) \theta(g)^\ast||_{\mathrm{op}} ||q||_{\mathrm{op}} \label{eq.repair-1-1.5} \\ & \leq 2L_{r,R}^2 ||I-\theta(g)\theta(g)^\ast||_{\mathrm{op}} \label{eq.repair-1-2} \\ & \leq 4L_{r,R}^2K_r\delta\label{eq.repair-1-3}  \end{align} 

Here, \begin{itemize} \item (\ref{eq.repair-1-1.5}) follows from (\ref{eq.repair-1-1}) by (\ref{eq.becky-1}) since $e^{2s_\delta} \leq 2$, \item (\ref{eq.repair-1-2}) follows from (\ref{eq.repair-1-1.5}) by Proposition \ref{prop.spec} since $q$ is self-adjoint, \item and (\ref{eq.repair-1-3}) follows from (\ref{eq.repair-1-2}) by (\ref{eq.becky-13}). \end{itemize}

Let $z \in \mathbb{C} \setminus \mathrm{spec}(q)$. We compute \begin{align}& ||(q - zI)^{-1} \theta(g) - \theta(g)(q-zI)^{-1}||_{\mathrm{op}} \nonumber \\ & =||(q - zI)^{-1} \theta(g) - (q - zI)^{-1}(q-zI) \theta(g)(q-zI)^{-1}||_{\mathrm{op}} \nonumber \\ & \leq ||(q-zI)^{-1}||_{\mathrm{op}} || \theta(g) - (q-zI) \theta(g)(q-zI)^{-1}||_{\mathrm{op}} \nonumber \\ & = \frac{1}{\mathrm{dist}(z,\mathrm{spec}(q))} || \theta(g) - (q-zI)\theta(g)(q-zI)^{-1}||_{\mathrm{op}} \nonumber \\ & = \frac{1}{\mathrm{dist}(z,\mathrm{spec}(q))}  \Bigl \vert \Bigl \vert \theta(g) - q\theta(g)(q-zI)^{-1} + z\theta(g)(q-zI)^{-1} \Bigr \vert \Bigr \vert_{\mathrm{op}} \nonumber  \\ & = \frac{1}{\mathrm{dist}(z,\mathrm{spec}(q))}  \Bigl \vert \Bigl \vert \theta(g) - \theta(g)q(q-zI)^{-1} + z\theta(g)(q-zI)^{-1}  \nonumber \\ & \hspace{2 in} +\theta(g)q(q-zI)^{-1} - q\theta(g)(q-zI)^{-1}\Bigr \vert \Bigr \vert_{\mathrm{op}} \nonumber  \\ & = \frac{1}{\mathrm{dist}(z,\mathrm{spec}(q))}  \Bigl \vert \Bigl \vert \theta(g)  - \theta(g)(q-zI)(q-zI)^{-1}+ \theta(g)q(q-zI)^{-1} -q\theta(g)(q-zI)^{-1}\Bigr \vert \Bigr \vert_{\mathrm{op}} \nonumber  \\ & = \frac{1}{\mathrm{dist}(z,\mathrm{spec}(q))}  ||\theta(g)q(q-zI)^{-1} - q\theta(g)(q-zI)^{-1}||_{\mathrm{op}} \nonumber  \\ & \leq \frac{1}{\mathrm{dist}(z,\mathrm{spec}(q))}  ||\theta(g)q- q\theta(g)||_{\mathrm{op}} ||(q-zI)^{-1}||_{\mathrm{op}} \nonumber  \\ & = \frac{1}{\mathrm{dist}(z,\mathrm{spec}(q))^2}  ||\theta(g)q- q\theta(g)||_{\mathrm{op}}  \label{eq.text-1} \\ & \leq \frac{4L_{r,R}^2K_r\delta}{\mathrm{dist}(z,\mathrm{spec}(q))^2}  \label{eq.text-2} \end{align}

Here, (\ref{eq.text-2}) follows from (\ref{eq.text-1}) by (\ref{eq.repair-1-3}). Now, let $c:[0,1] \to \mathbb{C}$ be a simple closed contour with the following properties. \begin{enumerate}[label=(\roman*)] \item \label{contour-1} We have $\mathrm{Re}(c(x)) > 0$ for all $x \in [0,1]$. \item \label{contour-2} The interval $[L_{r,R}^{-2},L_{r,R}^2]$ is enclosed by $c$. \item \label{contour-3} We have $\mathrm{dist}(c(x),[L_{r,R}^{-2},L_{r,R}^2]) \geq \frac{1}{2}L_{r,R}^{-2}$ for all $x \in [0,1]$. \item \label{contour-4} We have $\sup \{ |c(x)|: x \in [0,1] \} \leq 2L_{r,R}^2$  \item \label{contour-5} We have $\mathrm{\ell}(c) \leq 10L_{r,R}^2$ where $\ell(c)$ denotes the length of $c$. \end{enumerate} 

By Clause \ref{contour-1} we can consistently define a square root function on the image of $c$. Proposition \ref{prop.spec} together with Clause \ref{contour-2} in the definition of $c$ implies that $c$ encloses $\mathrm{spec}(q)$. Therefore we can use the holomorphic functional calculus to make the following computation.
 
 \begin{align} ||\theta(g)q^{\frac{1}{2}} &- q^{\frac{1}{2}}\theta(g)||_{\mathrm{op}} \nonumber \\ & = \frac{1}{2 \pi} \nml \theta(g) \left( \int_0^1 c(x)^{\frac{1}{2}}(c(x)I - q)^{-1} \deee x \right) - \left( \int_0^1 c(x)^{\frac{1}{2}}(c(x)I - q)^{-1} \deee x \right) \theta(g) \nmr_{\mathrm{op}} \nonumber \\ & = \frac{1}{2 \pi}  \nml \int_0^1 c(x)^{\frac{1}{2}} \Bigl(\theta(g)(c(x)I - q)^{-1} - (c(x)I - q)^{-1}\theta(g) \Bigr) \deee x  \nmr_{\mathrm{op}} \nonumber  \\ & \leq \frac{\ell(c)}{2\pi}  \sup_{0 \leq x \leq 1} \left( |c(x)|^{\frac{1}{2}} \Bigl \vert \Bigl \vert  \theta(g)(c(x)I - q)^{-1} - (c(x)I - q)^{-1}\theta(g) \Bigr \vert \Bigr \vert_{\mathrm{op}} \right) \label{eq.repair7}  \\ & \leq 10L_{r,R}^2 \sup_{0 \leq x \leq 1} \left( |c(x)|^{\frac{1}{2}} \Bigl \vert \Bigl \vert  \theta(g)(c(x)I - q)^{-1} - (c(x)I - q)^{-1}\theta(g) \Bigr \vert \Bigr \vert_{\mathrm{op}} \right) \label{eq.repair8} \\ & \leq 20L_{r,R}^3 \sup_{0 \leq x \leq 1} \Bigl \vert \Bigl \vert  \theta(g)(c(x)I - q)^{-1} - (c(x)I - q)^{-1}\theta(g) \Bigr \vert \Bigr \vert_{\mathrm{op}} \label{eq.repair9} \\ & \leq \frac{80L_{r,R}^5K_r\delta}{\mathrm{dist}(z,\mathrm{spec}(q))^2} \label{eq.repair10} \\ & \leq 320L_{r,R}^9K_r\delta \label{eq.repair11} \end{align}

Here, \begin{itemize} \item (\ref{eq.repair8}) follows from (\ref{eq.repair7}) by Clause \ref{contour-5} in the definition of $c$, \item (\ref{eq.repair9}) follows from (\ref{eq.repair8}) by Clause \ref{contour-4} in the definition of $c$, \item (\ref{eq.repair10}) follows from (\ref{eq.repair9}) by (\ref{eq.text-2}), \item and (\ref{eq.repair11}) follows from (\ref{eq.repair10}) by Clause \ref{contour-3} in the definition of $c$. \end{itemize}

Now, since $\mathrm{spec}(q) \subseteq [L_{r,R}^{-2},L_{r,R}^2]$, the spectral mapping theorem implies that $\mathrm{spec}(q^{-\frac{1}{2}}) \subseteq [L_{r,R}^{-1},L_{r,R}]$. Since $q^{-\frac{1}{2}}$ is self-adjoint, this implies $||q^{-\frac{1}{2}}||_{\mathrm{op}} \leq L_{r,R}$. Therefore \begin{align*} ||\zeta_\tril(g) - \theta(g)||_{\mathrm{op}} & = ||q^{\frac{1}{2}} \theta(g) q^{-\frac{1}{2}} - \theta(g)||_{\mathrm{op}} \\ & = ||q^{\frac{1}{2}} \theta(g) q^{-\frac{1}{2}} -  \theta(g)q^{\frac{1}{2}}q^{-\frac{1}{2}} ||_{\mathrm{op}} \\ & \leq  ||q^{\frac{1}{2}}\theta(g) - \theta(g)q^{\frac{1}{2}}||_{\mathrm{op}} ||q^{-\frac{1}{2}}||_{\mathrm{op}} \\ & \leq  320L_{r,R}^{10}K_r\delta \end{align*} Therefore Proposition \ref{prop.repair1} follows from (\ref{eq.numerical-1}) \end{proof}

\subsection{Finding a witness vector}

Define a vector $y \in \mathscr{Y}$ by setting \[ y = \frac{1}{d!} \bigoplus_{\varsigma \in \mathrm{Sym}(d)} \sum_{j=1}^d \alpha_{\varsigma j} \Phi_{\mathsf{D}_\varsigma}(e)_j \]   Since each $\mathsf{D}_\varsigma$ is normalized we have from (\ref{eq.alphal2}) that $y$ is a unit vector. Let $g \in \mathbb{B}_{r,\tril}$ and let $g' \in \mathbb{B}_{r,\trir}$. From Proposition \ref{prop.repair1} we have \begin{equation}  \left \langle \zeta(g,g') y, y \right \rangle  = \left \langle  \zeta_\tril(g) \zeta_\trir(g') y, y \right \rangle \approx_\epsilon \left \langle \theta(g)\zeta_\trir(g') y, y \right \rangle \label{eq.kwerpo} \end{equation} 

We have \begin{align} \left \langle \theta(g)\zeta_\trir(g') y, y \right \rangle  &= \frac{\epsilon}{d!}  \left \langle  \theta(g) \zeta_\trir(g')    \bigoplus_{\varsigma \in \mathrm{Sym}(d)} \sum_{j=1}^d \alpha_{\varsigma j} \Phi_{\hat{\mathsf{D}}_\circ}(e)_j , \bigoplus_{\varsigma \in \mathrm{Sym}(d)} \sum_{k=1}^d \alpha_{\varsigma k} \Phi_{\hat{\mathsf{D}}_\circ}(e)_k \right \rangle \nonumber \\ &+  \frac{1-\epsilon}{d!}  \left \langle  \theta(g) \zeta_\trir(g')    \bigoplus_{\varsigma \in \mathrm{Sym}(d)} \sum_{j=1}^d \alpha_{\varsigma j} \Phi_{\hat{\mathsf{D}}_\varsigma}(e)_j , \bigoplus_{\varsigma \in \mathrm{Sym}(d)} \sum_{k=1}^d \alpha_{\varsigma k} \Phi_{\hat{\mathsf{D}}_\varsigma}(e)_k \right \rangle \label{eq.bebe-1}  \end{align}

We have \begin{equation} \label{eq.bebe-2}  \frac{\epsilon}{d!} \left \vert \left \langle  \theta(g) \zeta_\trir(g')    \bigoplus_{\varsigma \in \mathrm{Sym}(d)} \sum_{j=1}^d \alpha_{\varsigma j} \Phi_{\hat{\mathsf{D}}_\circ}(e)_j , \bigoplus_{\varsigma \in \mathrm{Sym}(d)} \sum_{k=1}^d \alpha_{\varsigma k} \Phi_{\hat{\mathsf{D}}_\circ}(e)_k \right \rangle \right \vert \leq \epsilon \end{equation}

From (\ref{eq.kwerpo}), (\ref{eq.bebe-1}) and (\ref{eq.bebe-2}) we have

\begin{equation}   \left \langle \zeta(g,g') y, y \right \rangle \approx_\epsilon \frac{1}{d!}  \left \langle  \theta(g) \zeta_\trir(g')    \bigoplus_{\varsigma \in \mathrm{Sym}(d)} \sum_{j=1}^d \alpha_{\varsigma j} \Phi_{\hat{\mathsf{D}}_\varsigma}(e)_j , \bigoplus_{\varsigma \in \mathrm{Sym}(d)} \sum_{k=1}^d \alpha_{\varsigma k} \Phi_{\hat{\mathsf{D}}_\varsigma}(e)_k \right \rangle \label{eq.bebe-3} \end{equation} 

By construction we have \begin{equation} \zeta_\trir(g')  \Phi_{\hat{\mathsf{D}}_\varsigma}(e)_j = \Phi_{\hat{\mathsf{D}}_\varsigma}(g')_j \label{eq.28-3} \end{equation}

We have 

\begin{align} \frac{1}{d!} & \left \langle  \theta(g)   \bigoplus_{\varsigma \in \mathrm{Sym}(d)} \sum_{j=1}^d \alpha_{\varsigma j} \Phi_{\hat{\mathsf{D}}_\varsigma}(g')_j , \bigoplus_{\varsigma \in \mathrm{Sym}(d)} \sum_{k=1}^d \alpha_{\varsigma k} \Phi_{\hat{\mathsf{D}}_\varsigma}(e)_k \right \rangle \nonumber \\ &\hspace{1 in} =  \frac{1}{d!} \sum_{\varsigma \in \mathrm{Sym}(d)} \left \langle   \sum_{j=1}^d \alpha_{\tau(g)^{-1}\varsigma j} \Phi_{\hat{\mathsf{D}}_{\varsigma}}(g')_j , \sum_{k=1}^d \alpha_{\varsigma k} \Phi_{\hat{\mathsf{D}}_\varsigma}(e)_k \right \rangle \label{eq.26-1} \end{align}

From (\ref{eq.bebe-3}), (\ref{eq.28-3}) and (\ref{eq.26-1}) we obtain  

\begin{equation} \left \langle \zeta(g,g') y, y \right \rangle \approx_\epsilon  \frac{1}{d!}  \sum_{\varsigma \in \mathrm{Sym}(d)}  \sum_{j,k=1}^d \alpha_{\tau(g)^{-1}\varsigma j} \ov{ \alpha_{\varsigma k}} \left \langle  \Phi_{\hat{\mathsf{D}}_{\varsigma}}(g')_j , \Phi_{\hat{\mathsf{D}}_\varsigma}(e)_k  \right \rangle  \label{eq.28-7} \end{equation}

Since $g' \in \mathbb{B}_r$ from (\ref{eq.becky-5}) we have

\begin{equation}  \frac{1}{d!}  \sum_{\varsigma \in \mathrm{Sym}(d)}  \sum_{j,k=1}^d \alpha_{\sigma(g)^{-1}\varsigma j} \ov{ \alpha_{\varsigma k}} \hat{\mathsf{D}}_\varsigma(g')_{j,k} \approx_\epsilon  \frac{1}{d!}  \sum_{\varsigma \in \mathrm{Sym}(d)}  \sum_{j,k=1}^d \alpha_{\sigma(g)^{-1}\varsigma j} \ov{ \alpha_{\varsigma k}} \mathsf{D}_\varsigma(g')_{j,k} \label{eq.bebe-21} \end{equation}

From (\ref{eq.28-7}) and (\ref{eq.bebe-21}) we have

\begin{equation}  \langle \zeta(g,g') y, y  \rangle \approx_{2\epsilon} \frac{1}{d!}  \sum_{\varsigma \in \mathrm{Sym}(d)}  \sum_{j,k=1}^d \alpha_{\sigma(g)^{-1}\varsigma j} \ov{ \alpha_{\varsigma k}} \mathsf{D}_\varsigma(g')_{j,k} \label{eq.money-40} \end{equation}

From the construction of $\mathsf{D}$ we have

\begin{equation} \frac{1}{d!}  \sum_{\varsigma \in \mathrm{Sym}(d) \setminus B}  \sum_{j,k=1}^d \alpha_{\sigma(g)^{-1}\varsigma j} \ov{ \alpha_{\varsigma k}} \mathsf{D}_\varsigma(g')_{j,k} \approx_\epsilon \frac{1}{d!}  \sum_{\varsigma \in \mathrm{Sym}(d)}  \sum_{j,k=1}^d \alpha_{\sigma(g)^{-1}\varsigma j} \ov{ \alpha_{\varsigma k}} \mathsf{C}_\varsigma(g')_{j,k} \label{eq.bebe-24} \end{equation}

From (\ref{eq.money-40}) and (\ref{eq.bebe-24}) we have

\begin{equation}  \left \langle \zeta(g,g') y, y \right \rangle \approx_{3\epsilon} \frac{1}{d!}  \sum_{\varsigma \in \mathrm{Sym}(d) \setminus B}  \sum_{j,k=1}^d \alpha_{\sigma(g)^{-1}\varsigma j} \ov{ \alpha_{\varsigma k}} \mathsf{C}_\varsigma(g')_{j,k} \label{eq.bebe-25} \end{equation}

We have

\begin{align}&  \frac{1}{d!}  \sum_{\varsigma \in \mathrm{Sym}(d)}  \sum_{j,k=1}^d \alpha_{\sigma(g)^{-1}\varsigma j} \ov{ \alpha_{\varsigma k}} \mathsf{C}_\varsigma(g')_{j,k}  = \frac{1}{d!}  \sum_{\varsigma \in \mathrm{Sym}(d)}  \sum_{j,k=1}^d \alpha_{\sigma(g)^{-1}\varsigma j} \ov{ \alpha_{\varsigma k}}\left \langle \rho_\trir(g') x_{\varsigma j},x_{\varsigma k} \right \rangle \label{eq.money-41} \\ & \hspace{1 in} = \frac{1}{d!} \sum_{\varsigma \in \mathrm{Sym}(d)} \left \langle \rho_\trir(g') \sum_{j=1}^d \alpha_{\sigma(g)^{-1}\varsigma j} x_{\varsigma j}, \sum_{k=1}^d \alpha_{\varsigma k} x_{\varsigma k} \right \rangle \label{eq.bebe-28} \end{align}

where the equality in (\ref{eq.money-41}) holds by (\ref{eq.1-1}). From (\ref{eq.bebe-25}) and (\ref{eq.bebe-28}) we have

\begin{equation}   \langle \zeta(g,g') y, y \rangle \approx_{3\epsilon} \frac{1}{d!} \sum_{\varsigma \in \mathrm{Sym}(d)} \left \langle \rho_\trir(g') \sum_{j=1}^d \alpha_{\sigma(g)^{-1}\varsigma j} x_{\varsigma j}, \sum_{k=1}^d \alpha_{\varsigma k} x_{\varsigma k} \right \rangle \label{eq.money-50} \end{equation}

By making the changes of variables $j \mapsto \varsigma^{-1}\sigma(g)j$ in the left sum and $k \mapsto \varsigma^{-1}k$ in the right sum of (\ref{eq.money-50}) we obtain \[ (\ref{eq.money-50}) = \frac{1}{d!} \sum_{\varsigma \in \mathrm{Sym}(d)} \left \langle \rho_\trir(g') \sum_{j=1}^d \alpha_{j} x_{\sigma(g)j}, \sum_{k=1}^d \alpha_k x_k \right \rangle \]

or equivalently

\begin{equation} (\ref{eq.money-50})  =\left \langle \rho_\trir(g') \sum_{j=1}^d \alpha_{j} x_{\sigma(g)j}, \sum_{k=1}^d \alpha_k x_k \right \rangle \label{eq.bebe-29.5} \end{equation}

From (\ref{eq.money-50}) and (\ref{eq.bebe-29.5}) we obtain

\begin{equation}  \left \langle \zeta(g,g') y, y \right \rangle \approx_{3\epsilon} \left \langle \rho_\trir(g') \sum_{j=1}^d \alpha_{j} x_{\sigma(g)j}, \sum_{k=1}^d \alpha_k x_k \right \rangle \label{eq.bebe-30} \end{equation}

From (\ref{eq.uac-k}) we have 

\begin{equation} \left \langle \rho_\trir(g') \sum_{j=1}^d \alpha_{j} x_{\sigma(g)j}, \sum_{k=1}^d \alpha_k x_k \right \rangle= \left \langle \rho_\trir(g') \kappa(g)  \sum_{j=1}^d \alpha_{j} x_j, \sum_{k=1}^d \alpha_k x_k \right \rangle \label{eq.bebe-31} \end{equation}

From (\ref{eq.uac-g}) we have

\begin{equation} \left \langle \rho_\trir(g') \kappa(g)  \sum_{j=1}^d \alpha_{j} x_j, \sum_{k=1}^d \alpha_k x_k \right \rangle \approx_\epsilon \left \langle \rho_\trir(g') \rho_\tril(g)  \sum_{j=1}^d \alpha_{j} x_j, \sum_{k=1}^d \alpha_k x_k \right \rangle \label{eq.bebe-32} \end{equation}

From (\ref{eq.bebe-30}), (\ref{eq.bebe-31}) and (\ref{eq.bebe-32}) we have

\begin{equation}  \left \langle \zeta(g,g') y, y \right \rangle \approx_{4\epsilon} \left \langle \rho_\trir(g') \rho_\tril(g)  \sum_{j=1}^d \alpha_{j} x_j, \sum_{k=1}^d \alpha_k x_k \right \rangle \label{eq.bebe-33} \end{equation}

From (\ref{eq.uac-x}) and (\ref{eq.bebe-33}) we obtain

\[  \langle \zeta(g,g') y, y \rangle \approx_{5\epsilon} \langle \rho(g,g') x, x \rangle \]

This completes the proof of Conjecture \ref{thm.half}. In combination with the arguments of Subsection \ref{subsec.diablo} this completes the proof of Conjecture \ref{thm.CEC}.

\bibliographystyle{plain}
\bibliography{2019.12.20-conjecture}

\begin{thebibliography}{1}

\bibitem{MR2316876}
M.~Bakonyi and D.~Timotin.
\newblock Extensions of positive definite functions on free groups.
\newblock {\em J. Funct. Anal.}, 246(1):31--49, 2007.

\bibitem{MR2415834}
Bachir Bekka, Pierre de~la Harpe, and Alain Valette.
\newblock {\em Kazhdan's property ({T})}, volume~11 of {\em New Mathematical
  Monographs}.
\newblock Cambridge University Press, Cambridge, 2008.

\bibitem{MR2391387}
Nathanial~P. Brown and Narutaka Ozawa.
\newblock {\em {$C^*$}-algebras and finite-dimensional approximations},
  volume~88 of {\em Graduate Studies in Mathematics}.
\newblock American Mathematical Society, Providence, RI, 2008.

\bibitem{MR1402012}
Kenneth~R. Davidson.
\newblock {\em {$C^*$}-algebras by example}, volume~6 of {\em Fields Institute
  Monographs}.
\newblock American Mathematical Society, Providence, RI, 1996.

\bibitem{MR1468230}
Richard~V. Kadison and John~R. Ringrose.
\newblock {\em Fundamentals of the theory of operator algebras. {V}ol. {II}},
  volume~16 of {\em Graduate Studies in Mathematics}.
\newblock American Mathematical Society, Providence, RI, 1997.
\newblock Advanced theory, Corrected reprint of the 1986 original.

\bibitem{MR2583950}
Alexander~S. Kechris.
\newblock {\em Global aspects of ergodic group actions}, volume 160 of {\em
  Mathematical Surveys and Monographs}.
\newblock American Mathematical Society, Providence, RI, 2010.

\bibitem{MR2931406}
Alexander~S. Kechris.
\newblock Weak containment in the space of actions of a free group.
\newblock {\em Israel J. Math.}, 189:461--507, 2012.

\bibitem{MR2077037}
Alexander Lubotzky and Yehuda Shalom.
\newblock Finite representations in the unitary dual and {R}amanujan groups.
\newblock In {\em Discrete geometric analysis}, volume 347 of {\em Contemp.
  Math.}, pages 173--189. Amer. Math. Soc., Providence, RI, 2004.

\bibitem{MR3067294}
Narutaka Ozawa.
\newblock About the {C}onnes embedding conjecture: algebraic approaches.
\newblock {\em Jpn. J. Math.}, 8(1):147--183, 2013.

\end{thebibliography}

Department of Mathematics\\
The University of Texas at Austin\\
Austin, TX 78712\\
\\
\texttt{pjburton@math.utexas.edu}\\
\texttt{kate.juschenko@gmail.com}

\end{document}